\documentclass[12 pt]{article}%
\usepackage{amsmath, amsfonts, amsthm, color,latexsym}
\usepackage{amsthm, amsfonts, amscd, amsmath, amssymb,color,soul}
\usepackage{amsmath}
\usepackage{amsfonts}
\usepackage{amssymb}
\usepackage{graphicx}%
\providecommand{\U}[1]{\protect\rule{.1in}{.1in}}
\newtheorem{theorem}{Theorem}[section]
\newtheorem{proposition}[theorem]{Proposition}

\newtheorem{example}[theorem]{Example}

\newtheorem{remark}[theorem]{Remark}

\newtheorem{lemma}[theorem]{Lemma}
\newtheorem{final remark}[theorem]{Final Remark}
\newtheorem{definition}[theorem]{Definition}
\begin{document}

\title{\sc On Pietsch measures for summing operators and dominated polynomials}
\date{}
\author{Geraldo Botelho\thanks{Supported by CNPq Grant
302177/2011-6 and Fapemig Grant PPM-00295-11.}\,\,, Daniel Pellegrino\thanks{Supported by CNPq Grant
301237/2009-3.}~~and Pilar Rueda\thanks{Supported by Ministerio de Ciencia e Innovaci\'on MTM2011-22417.\hfill\newline2010 Mathematics Subject
Classification. 28C15, 46G25, 47B10, 47L22. \newline Keywords: absolutely summing operators, dominated homogeneous polynomials, Pietsch measures, factorization theorem.}}
\maketitle

\begin{abstract} We relate the injectivity of the canonical map from $C(B_{E'})$ to $ L_p(\mu)$, where $\mu$ is a regular Borel probability measure on the closed unit ball $B_{E'}$ of the dual $E'$ of a Banach space $E$ endowed with the weak* topology, to the existence of injective $p$-summing linear operators/$p$-dominated homogeneous polynomials defined on $E$ having $\mu$ as a Pietsch measure. As an application we fill the gap in the proofs of some results of \cite{BoPeRu, BoPeRuLAMA} concerning Pietsch-type factorization of dominated polynomials.

\end{abstract}

\section{Introduction}
The Pietsch domination theorem and the Pietsch factorization theorem are two cornerstones in the theory of absolutely $p$-summing linear operators. The domination theorem asserts that a linear operator $u \colon E \longrightarrow F$ between Banach spaces is absolutely $p$-summing if and only if there is a constant $C>0$ and
 a regular Borel probability measure on the closed unit ball $B_{E'}$ of the dual $E'$ of $E$ endowed with the weak* topology such that
\begin{equation}\label{dominationlinear}
\|u(x)\| \leq C \left(\int_{B_{E'}} |\varphi(x)|^p d\mu(\varphi)\right)^{\frac{1}{p}}{\rm ~ for~ every~} x \in E.
\end{equation}
Any such measure $\mu$ is called a {\it Pietsch measure} for $u$.

The first attempts to generalize the theory to the nonlinear setting led to classes of multilinear mappings \cite{Pie, geiss} and homogeneous polynomials \cite{ma96} that enjoy a Pietsch-type domination theorem. And that is why these mappings are called {\it dominated}. By definition, a continuous $n$-homogeneous $P\colon E \longrightarrow F$ is $p$-dominated if $P$ sends weakly $p$-summable sequences in $E$ to absolutely $\frac{p}{n}$-summable sequences in $F$. In \cite{ma96} it is proved that $P$ is $p$-dominated if and only if there is a constant $C>0$ and
 a regular Borel probability measure on $B_{E'}$ endowed with the weak* topology such that
\begin{equation}\label{dominationpolynomial}
\|P(x)\| \leq C \left(\int_{B_{E'}} |\varphi(x)|^p d\mu(\varphi)\right)^{\frac{n}{p}}{\rm ~ for~ every~} x \in E.
\end{equation}
Again, any such measure $\mu$ is called a {\it Pietsch measure} for $P$.

Continuing this line of thought, the last decades witnessed the emergence of a great variety of classes of nonlinear operators, such as multilinear mappings, homogeneous polynomials, subhomogeneous mappings, Lipschitz mappings and, quite recently,  arbitrary mappings, that generalize the notion of absolutely summing linear operator. In many cases, the researchers end up with mappings that enjoy a Pietsch-type domination theorem. Very general versions of the domination theorem for summing nonlinear mappings were successfully treated in \cite{BoPeRuJMAA, PeSa, PeSaSe}.

On the other hand, the story of the Pietsch factorization theorem is quite different. Let us recall the linear case. Given a Banach space $E$, consider the
linear isometry
$$e \colon E \longrightarrow C(B_{E'})~,~ e(x)(\varphi) = \varphi(x).$$
Given $0 < p < + \infty$ and a regular Borel probability measure $\mu$ on $B_{E'}$, let $j_p \colon C(B_{E'})
\longrightarrow L_p(\mu)$ be the canonical map. By $j_p^e$ we denote the restriction of $j_p$ to $e(E)$. The Pietsch factorization theorem asserts that if $\mu$ is a Pietsch measure for the $p$-summing linear operator $u \colon E \longrightarrow F$, then there exist a (closed) subspace $X_p$ of $L_p(\mu)$ containing $(j_p \circ e)(E)$ and a continuous operator $\hat u \colon
  X_p\longrightarrow F$ such that the following diagram commutes:
  \begin{center}
$\begin{CD}
E                      @ > u >>            F\\
@V{e}VV                  @ AA\hat uA\\
e(E) @>{j_p^e}>> X_p\\
@VVV @VVV\\
C(B_{E'}) @>{j_p} >> L_p(\mu)
\end{CD}$
\end{center}
Here, $X_p$ is nothing but the completion of the linear space $j_p^e\circ e(E)$ and $\hat u$ is defined by $\hat u(j_p^e\circ e(x))=u(x)$, $x\in E$, and extended by continuity to the completion $X_p$. In other words, $j_p$ is a $p$-summing operator through which every $p$-summing operator factors.

Encouraged by the validity of the domination theorem for several classes of nonlinear summing mappings, nonlinear versions of the factorization theorem started been pursued. At this point it is important to observe how the factorization theorem is derived from the domination theorem. The linearity of $u$ jointly with the domination formula (\ref{dominationlinear}) yield the well definition of $\hat u$ and its continuity. Indeed, if $j_p^e\circ e(x)=j_p^e\circ e(y)$ for some $x, y\in E$ then $ \left(\int_{B_{E'}} |\varphi(x-y)|^p d\mu(\varphi)\right)^{\frac{1}{p}}=0$. Hence, by (\ref{dominationlinear}) we have $\|u(x)-u(y)\|=\|u(x-y)\|=0$ and then $u(x)=u(y)$. The role played by the linearity of $u$ must be noted. This reasoning obviously does not work to derive a factorization theorem for nonlinear summing mappings from a correspondent domination theorem. If the canonical map $j_p$ -- or its restriction $j_p^e$ -- is injective, then the conclusion is straightforward. In Section 2 we prove that $j_p^e$ is injective if and only if $\mu$ is a Pietsch measure for some injective $p$-summing linear operator defined on $E$. In Section 3 we investigate the validity of this equivalence for $p$-dominated polynomials. To do so we prove that a dominated polynomial shares some of its Pietsch measures with all its derivatives. The result proved in Section 2 is enough to give examples where $j_p^e$ is injective. However, we also prove that $j_p^e$ may fail to be injective. Therefore, the domination theorem (\ref{dominationpolynomial}) for polynomials does not suffice to ensure that the closing operator for a factorization scheme is well-defined. The conclusion is that the nonlinear case demands new techniques.

In \cite{BoPeRu} the authors used a renorming technique to state a Pietsch-type factorization theorem for dominated polynomials asserting that if $\mu$ if a Pietsch measure for the $p$-dominated $n$-homogeneous polynomial $P \colon E \longrightarrow F$, then there is a renorming of the subspace $(j_{p/n})^n \circ e(E)$ of $L_{p/n}(\mu)$ and an operator $v$ from the resulting space to $F$ such that $P = v \circ (j_{p/n})^n \circ e$. In this fashion $(j_{p/n})^n$ is a prototype of an $n$-homogeneous polynomial through which every $p$-dominated polynomial factors. This accomplishes the search for a Pietsch-type factorization theorem for dominated polynomials. Moreover, in \cite{BoPeRuLAMA} the authors give a reinterpretation of this factorization scheme in terms of symmetric projective tensor products. The problem is that the authors  did not take into account that the canonical map $j_p^e$ may fail to be injective, and the proofs of some of the results in \cite{BoPeRu, BoPeRuLAMA} use the injectivity of $j_p^e$. In Section 4 we apply the results of Section 3 to provide these result of \cite{BoPeRu, BoPeRuLAMA} (in the normed case $p \geq n$) with proofs that do not depend on the injectivity of $j_p^e$. So, although there is a gap in the original proofs, the results of \cite{BoPeRu, BoPeRuLAMA} hold in the normed case as they are stated there up to a slight modification in the renorming.
%
%
%
%
%
%
%

 Given a Banach space $E$, let $E'$ denote its topological dual space. The closed unit ball $B_{E'}$ of $E'$ is henceforth endowed with the weak* topology. 
The space of all  continuous $n$-homogeneous polynomials between Banach spaces $E$ and $F$ is denoted by ${\mathcal P}(^nE;F)$ and  is endowed with the usual sup norm. For background on spaces of polynomials we refer to \cite{din, mujica}.  

\section{The canonical map $C(B_{E'}) \longrightarrow L_p(\mu)$}
In this section we establish that the canonical map $j_p \colon  C(B_{E'}) \longrightarrow L_p(\mu)$ is sometimes injective and sometimes non-injective.

First we give some examples where $j_p$ is injective. Remember that $j_p^e$ denotes the restriction of $j_p$ to $e(E)$ and that $e \colon E \longrightarrow C(B_{E'})$ is an isometry, hence injective. By $W(B_{E'},w^*)$ we mean the set of all regular Borel probability measures on $B_{E'}$ endowed with the weak* topology.

\begin{proposition}\label{TH} Let $\mu \in  W(B_{E'},w^*)$ and $1 \leq p < \infty$ be given. Then  $j_p^e$ is injective if and only if $\mu$ is a Pietsch measure for some injective $p$-summing linear operator defined on $E$.   \end{proposition}

\begin{proof} If $j_p^e \colon e(E) \longrightarrow L_p(\mu)$ is injective, then $j_p^e \circ e \colon E \longrightarrow L_p(\mu)$ is an injective $p$-summing linear operator on $E$ having $\mu$ as a Pietsch measure.

Conversely, assume that there are a Banach space $F$ and an injective $p$-summing linear operator $u \colon E \longrightarrow F$ having $\mu$ as a Pietsch measure. By the Pietsch Factorization Theorem there exists a continuous linear operator $v \colon j_p(e(E)) \longrightarrow F$ such that $u = v \circ j_p^e \circ e$. Let $x,y \in E$ be such that $j_p^e(e(x)) = j_p^e(e(y))$. Then,
$$ u(x) = v(j_p^e(e(x))) = v(j_p^e(e(y))) = u(y).$$
The injectivity of $u$ gives $x = y$, proving that $j_p^e$ is injective.
\end{proof}

\begin{example} \rm (i) It is clear that the formal inclusion $i \colon \ell_1 \longrightarrow \ell_2$ is an injective linear operator. It is 1-summing by Grothendieck's Theorem. So, if $\mu$ is a Pietsch measure for $i$, then by Proposition \ref{TH} the canonical map $j_1^e \colon e(\ell_1) \longrightarrow L_1(\mu)$ is injective.

\medskip

\noindent (ii) Let $\mu$ be a finite measure. The formal inclusion $I \colon L_\infty(\mu) \longrightarrow L_p(\mu)$ is clearly an injective linear operator. It is $p$-summing by \cite[Example 2.9(d)]{DJT}. So, if $\nu$ is a Pietsch measure for $I$, then by Proposition \ref{TH} the canonical map $j_p^e \colon e(L_\infty(\mu)) \longrightarrow L_p(\nu)$ is injective. It is well known that $L_\infty(\mu)$ is a $C(K)$ space for some compact Hausdorff space $K$, so $e\colon L_\infty(\mu) \longrightarrow C(B_{L_\infty(\mu)'})$ is bijective (see \cite[proof of Corollary 2.15]{DJT}). It follows that $j_p \colon C(B_{L_\infty(\mu)'}) \longrightarrow L_p(\mu)$ is injective.

\medskip

\noindent (iii) It is clear that
$$u \colon c_0 \longrightarrow \ell_1~,~u\left((a_j)_{j=1}^\infty \right)= \left(\frac{a_j}{2^j} \right)_{j=1}^\infty,  $$
is an injective continuous linear operator. It is 2-summing by the dual Grothendieck theorem \cite[Theorem 3.7]{DJT}. So, if $\mu$ is a Pietsch measure for $u$, then by Proposition \ref{TH} the canonical map $j_2^e \colon e(c_0) \longrightarrow L_2(\mu)$ is injective. As in example (ii) it follows that $j_p \colon C(B_{\ell_1}) \longrightarrow L_p(\mu)$ is injective.
\end{example}

Now we show that sometimes $j_p$ fails to be injective.

\begin{proposition} Let $K$ be a compact Hausdorff space containing at least two elements and  $0 < p < \infty$. Then there is a regular probability measure $\mu$ on the Borel sets of $K$ such that the canonical mapping $j_p \colon C(K) \longrightarrow L_p(\mu)$ fails to be injective.
\end{proposition}

\begin{proof} Pick $x,z \in K$, $x \neq z$. There are open sets $V,W$ such that $x \in V$, $z \in W$ and $V \cap W = \emptyset$. Then $V^c$ is closed, $x \notin V^c$ and $W \subseteq V^c$. It follows that
$$x \notin \textstyle\bigcap\{F : F {\rm ~is~closed~and~} W \subseteq F\} = \overline{W}.  $$
So, $\overline{W}$ and $\{x\}$ are disjoint closed subsets of the normal space $K$ (remember that compact Hausdorff spaces are normal). By Urysohn's Lemma there is a continuous function $f \colon K \longrightarrow [0,1]$ such that $f(x) =1$ and $f(y) = 0$ for every $y \in \overline{W}$. In particular, $f \neq 0$ in $C(K)$. Let $\mu$ be the restriction of the Dirac measure at $z$ (the point mass associated with $z$) to the Borel sets of $K$. Since $z \in W \subseteq \overline{W}$, we have $z \notin \overline{W}^{\,c}$, so $\mu( \overline{W}^{\,c}) = 0$. Then $f(y) = 0$ for every $y \in \overline{W}$ and $\mu( \overline{W}^{\,c}) = 0$. This shows that $f = 0$ $\mu$-almost everywhere, that is, $j_p(f) = 0$ in $L_p(\mu)$.
\end{proof}

\bigskip

For every Banach space $E\neq \{0\}$, $(B_{E'}, w^*)$ is a compact Hausdorff space containing more than two elements, so there is always a regular probability measure $\mu$ on $(B_{E'}, w^*)$ such that $j_p \colon C(B_{E'}) \longrightarrow L_p(\mu)$ fails to be injective. As we have seen above, $e$ is bijective whenever $E$ is a $C(K)$ space, so in this case $j_p^e$ is not injective either.

\begin{remark}\rm Although not injective in general, the map $j_p$ is sometimes referred to as {\it the natural injection} (see, e.g., \cite[p.\!\! 1301]{maurey}, \cite[p.\,9]{rosenthal}).
\end{remark}

\section{Injective dominated polynomials}
Considering that dominated polynomials have Pietsch measures as well as summing operators, it is a natural question to ask for a polynomial version of Proposition \ref{TH}: given $n \in \mathbb{N}$, is it true that $j_p ^ e$ is injective if and only if $\mu$ is a Pietsch measure for some injective $p$-dominated $n$-homogeneous polynomial defined on $E$?

For $n \geq 2$, when the scalars are complex or $n$ is even, $n$-homogeneous polynomials are never injective.
Therefore, the above question shall be treated only for real $n$-homogeneous polynomials with $n$ being odd. Remember that, for $n$ odd and real scalars, injective $n$-homogeneous polynomials on infinite dimensional spaces do exist: for example, for every $n$ odd, in the real case the $n$-homogeneous polynomial
$$(a_j)_{j=1}^\infty \in \ell_\infty \mapsto (a_j^n)_{j=1}^\infty \in \ell_\infty, $$
is injective.

It is well known that the derivatives of a $p$-dominated polynomial $P$ are $p$-dominated polynomials as well. But this does not mean that $P$ shares a Pietsch measure with its derivatives. This condition is crucial for our purposes. 
Given  $P\in {\cal P}(^nE;F)$, $a \in E$ and $k \in \{1, \ldots, n\}$, by $\widehat d^kP(a)$ we denote the $k$th derivative of $P$ at $a$ and by $\check P$ the unique symmetric $n$-linear mapping such that $P(x)=\check P(x,\ldots,x)$ for all $x\in E$. Observe that $P = \widehat d^nP(a)$.

\begin{definition}\rm A measure $\mu \in W(B_{E'},w^*)$ is said to be a {\it differential Pietsch measure} for the $p$-dominated polynomial $P\in {\cal P}(^nE;F)$ if $\mu$ is a Pietsch measure for $\widehat d^kP(a)$ for all $a \in E$ and $k = 1, \ldots, n$.
\end{definition}

First of all we have to establish that it is not restrictive to work with differential Pietsch measures:

\begin{proposition}\label{prop1} Every $p$-dominated homogeneous polynomial has a differential Pietsch measure.
\end{proposition}

\begin{proof} Let $P \in {\cal P}(^nE;F)$ be a $p$-dominated polynomial. By \cite[Proposition 3.4]{BoPeRu} there are a Banach space $G$, a $p$-summing linear operator $u \colon E \longrightarrow G$ and a continuous $n$-homogeneous polynomial $Q \colon G \longrightarrow F$ such that $P = Q \circ u$. Then $\check P = \check Q \circ (u, \ldots, u)$. Let $\mu$ be a Pietsch measure for $u$ with constant $C$. Then, for $a \in E$ and $k \in \{1, \ldots, n\}$,
\begin{align*} \|\widehat d^kP(a)(x)\| & = \left\|\frac{n!}{(n-k)!}\check P(a^{n-k}, x^k) \right\| =  \frac{n!}{(n-k)!}\cdot \left\|\check Q\left(u(a)^{n-k}, u(x)^k\right) \right\|\\
& \leq \frac{n!}{(n-k)!}\cdot \|\check Q\| \cdot \|u(a)\|^{n-k} \cdot \|u(x)\|^k \\
& \leq \frac{n!}{(n-k)!}\cdot \|\check Q\| \cdot \|u\|^{n-k}\cdot \|a\|^{n-k} \cdot  C^k \cdot \left(\int_{B_{E'}}|\varphi(x)|^p d\mu(\varphi)\right)^\frac{k}{p},
\end{align*}
proving that $\mu$ is a Pietsch measure for $\widehat d^kP(a) $.
\end{proof}


\begin{remark}\rm It is worth mentioning that the proof of \cite[Proposition 3.4]{BoPeRu} we are aware of, namely \cite[Proposition 46(a)]{note}, relies on \cite[Corolario 3.23]{david}, a result whose proof depends on the injectivity of $j_p$. Nevertheless, despite this gap in the proof, \cite[Corolario 3.23]{david} is correct. The following argument, kindly communicated to the authors by E. A.  S\'anchez-P\'erez, fulfills the gap in the proof. Let $E_1, \ldots, E_n, F$ be Banach spaces, $\mu_j \in W(B_{E_j'}, w^*)$ for $j = 1, \ldots, n$, $p_1, \ldots, p_n \geq 0$, and $A \colon E_1 \times \cdots \times E_n \longrightarrow F$ be an $n$-linear mapping for which there is a constant $C > 0$ such that
$$\|A(x_1, \ldots, x_n)\| \leq C \prod\limits_{j=1}^n \left( \int_{B_{E_j'}}|\varphi_j(x_j)|^{p_j}\,d\mu_j(\varphi_j) \right)^{1/p_j}  $$
for all $x_1 \in E_1, \ldots, x_n \in E_n$. We are supposed to show that if $x_j,y_j \in E_j$ for $j = 1, \ldots, n$, are such that $j_{p_j}(e(x_j)) = j_{p_j}(e(y_j))$ in $L_{p_j}(\mu_j)$ for $j = 1, \ldots, n$, then $A(x_1, \ldots, x_n) = A(y_1, \ldots, y_n)$. Indeed, the equality in $L_{p_j}(\mu_j)$ gives
$$\int_{B_{E_j'}} |e(x_j-y_j)(\varphi_j)|^{p_j}\,d\mu_j(\varphi_j) = \int_{B_{E_j'}} |\varphi_j(x_j-y_j)|^{p_j}\,d\mu_j(\varphi_j) = 0 $$
for $j = 1, \ldots, n$. So,
 $$ \|A(x_1, \ldots, x_n) - A(y_1, \ldots, y_n)\|~~~~~~~~~~~~~~~~~~~~~~~~~~~~~~~~~~~~~~~~~~~~~~~~~~~~~~~~~~~~~~~~~~~~~
 ~~~~~~~~~~~~~~~~~~~~~~~~~~~$$
 $$= \left\|\sum_{j=1}^n \left[A(y_1, \ldots, y_{j-1}, x_j ,\ldots, x_n) - A(y_1, \ldots, y_j, x_{j+1}, \ldots, x_n) \right]\right\|~~~~~~~~~~~~~~~~~~~~~~~~~~~~~~~~~~~~~~~~~~~~~~~~~~~~~~~~~~~~~~$$
$$= \left\|\sum_{j=1}^n A(y_1, \ldots, y_{j-1}, x_j-y_j , x_{j+1}\ldots, x_n)\right\|~~~~~~~~~~~~~~~~~~~~~~~~~~~~~~~~~~~~~~~~~~~~~~~~~~~~~~~~~~~~$$
$$\leq \sum_{j=1}^n \left\|A(y_1, \ldots, y_{j-1}, x_j-y_j , x_{j+1}\ldots, x_n)\right\|~~~~~~~~~~~~~~~~~~~~~~~~~~~~~~~~~~~~~~~~~~~~~~~~~~~~~~~~~~~~~~~~~~$$
$$\leq \sum_{j=1}^n C \left[\prod\limits_{k=1, k\neq j}^n \left( \int_{B_{E_k'}}|\varphi_k(z_k)|^{p_k}\,d\mu_k(\varphi_k)\right)^{1/p_k} \right] \cdot\left( \int_{B_{E_j'}} |\varphi_j(x_j-y_j)|^{p_j}\,d\mu_j(\varphi_j) \right)^{1/p_j} =0,$$
where $z_k=y_k$ if $k<j$ and $z_k=x_k$ if $k>j$.
\end{remark}

\bigskip

The next result shall be useful in this section and in Section 4 too.

\begin{proposition}\label{prop2} Let $\mu \in  W(B_{E'},w^*)$ and $0 < p < \infty$. If $x,y \in E$ are such that $j_p^e(e(x)) = j_p^e(e(y))$, then $P(x) = P(y)$ for every $p$-dominated polynomial $P \in {\cal P}(^nE;F)$ (regardless of the $n \in \mathbb{N}$ and the Banach space $F$) having $\mu$ as a differential Pietsch measure.
\end{proposition}

\begin{proof} The assumption $j_p^e(e(x)) = j_p^e(e(y))$ says that $e(x-y)$ is null $\mu$-almost everywhere, so
$$\int_{B_{E'}}|\varphi(x-y)|^p d\mu(\varphi) = \int_{B_{E'}}|e(x-y)(\varphi)|^p d\mu(\varphi) =0.$$
 Let $P \in {\cal P}(^nE;F)$ be a $p$-dominated polynomial having $\mu$ as a differential Pietsch measure. By definition we know that $\mu$ is a Pietsch measure for $\widehat d^kP(y)$ for every $k = 1, \ldots, n$, say with constant $C_k$.
By \cite[Lemma 1.9]{din} we know that
\begin{align*}
P(x)-P(y)&=\sum_{k=0}^{n-1} \binom{n}{k} \check{P}(y^k,(x-y)^{n-k}) = \sum_{k=0}^{n-1} \frac{\widehat d^{n-k}P(y)(x-y)}{(n-k)!}.
\end{align*}
Therefore,
\begin{align*} \|P(x) - P(y)\| &\leq \sum_{k=0}^{n-1} \frac{\|\widehat d^{n-k}P(y)(x-y)\|}{(n-k)!} \\&\leq \sum_{k=0}^{n-1}\frac{C_{n-k}}{(n-k)!}\cdot \left(\int_{B_{E'}}|\varphi(x-y)|^p d\mu(\varphi)\right)^\frac{n-k}{p}=0,
\end{align*}
proving that $P(x) = P(y)$.
\end{proof}

Now we can prove that if $j_p^e$ fails to be injective, then dominated homogeneous polynomials having $\mu$ as a differential Pietsch measure are never injective. As we remarked before, this result is of some interest only for $n$ odd and real scalars.

%
%
\begin{proposition} \label{propos} Let $n \in \mathbb{N}$ odd, $p \geq n$, $E$ be a real Banach space and $\mu \in W(B_{E'},w^*)$ be given. Then $j_p^e$ is injective if and only if $\mu$ is a differential Pietsch measure for some injective $p$-dominated $n$-homogeneous polynomial defined on $E$.
\end{proposition}

\begin{proof} Assume that $j_p^e$ is injective. Since $L_{p/n}(\mu)$ is a Banach space, we can consider the $n$-homogeneous polynomial
 $$(j_{p/n}^e )^n \colon e(E) \longrightarrow L_{p/n}(\mu)~,~(j_{p/n}^e )^n(f) = j_{p/n}^e(f^n). $$
Letting $i_{n,p} \colon L_p(\mu) \longrightarrow L_{p/n}(\mu)$ denote the formal inclusion, we have $(j_{p/n}^e )^n  = i_{n,p} \circ (j_{p}^e )^n $. Since $e, j_p^e, i_{n,p}$ are injective and $n$ is odd, it follows that the $n$-homogeneous polynomial $P:=(j_{p/n}^e )^n \circ e \colon E \longrightarrow L_{p/n}(\mu)$ is injective. Let us see that $\mu$ is a differential Pietsch measure for $P$. First observe that, for $x_1, \ldots, x_n \in E$,
$$\check{P}(x_1, \ldots, x_n) = j_{p/n}^e(e(x_1)\cdot e(x_2) \cdots e(x_n)). $$
Given $a \in E$ and $k \in \{1, \ldots, n\}$, since $\frac{kp}{n} \leq p$ we have
\begin{align*}\left\|\widehat d^k P(a)(x)\right\|_{p/n}&= \left\|\frac{n!}{(n-k)!}\check P(a^{n-k}, x^k) \right\|_{p/n}\\& = \frac{n!}{(n-k)!}\cdot \left(\int_{B_{E'}}\left| [j_{p/n}^e(e(a)^{n-k}\cdot e(x)^k)](\varphi)\right|^{p/n}d\mu(\varphi) \right)^{n/p}\\&
= \frac{n!}{(n-k)!}\cdot \left(\int_{B_{E'}}\left| \varphi(a)^{n-k}\cdot \varphi(x)^k\right|^{p/n}d\mu(\varphi) \right)^{n/p}\\&
\leq \frac{n!}{(n-k)!}\cdot \sup_{\varphi \in B_{E'}}|\varphi(a)|^{n-k}\cdot \left(\int_{B_{E'}}\left|\varphi(x)\right|^{kp/n}d\mu(\varphi) \right)^{n/p}\\&
\leq \frac{n!}{(n-k)!}\cdot \|a\|^{n-k}\cdot \left(\int_{B_{E'}}\left|\varphi(x)\right|^pd\mu(\varphi) \right)^{k/p},
\end{align*}
for every $x \in E$, proving that $\mu$ is a Pietsch measure for $\widehat d^k P(a)$.

Conversely, let  $F$ be a real Banach space and $P \in {\cal P}(^nE;F)$  be an injective $p$-dominated polynomial having $\mu$ as a differential Pietsch measure. If $x,y \in E$ are such that $j_p^e(e(x)) = j_p^e(e(y))$, by Proposition \ref{prop2} we have that $P(x) = P(y)$, and by the injectivity of $P$ it follows that $x = y$.
\end{proof}

\begin{remark}\rm Note that the condition $p \geq n$ can be dropped in the \textquotedblleft if\textquotedblright part of the proposition above.
\end{remark}


\section{Pietsch factorization of dominated polynomials}
In \cite{BoPeRu, BoPeRuLAMA} we state some results concerning Pietsch-type factorization of dominated polynomials. The proofs of \cite[Proposition 4.2, Theorem 4.4]{BoPeRu} and \cite[Theorem 2.1, Proposition 3.3]{BoPeRuLAMA} use, in one way or another, the injectivity of $j_p^e$. But, as we proved in Section 2, $j_p^e$ may fail to be injective. In this section we use Propositions \ref{prop1} and \ref{prop2} to provide these results with proofs that do not depend on the injectivity of $j_p^e$.

Let $P \in {\cal P}(^nE;F)$ be a $p$-dominated polynomial. The infimum of the constants $C$ satisfying (\ref{dominationpolynomial}) is denoted by $\|P\|_{d,p}$. It is well known that $\|\cdot \|_{d,p}$ is a norm if $p \geq n$ and a $\frac{p}{n}$-norm if $p < n$. In this section we restrict ourselves to the normed case. So, henceforth $n$ is a positive integer, $p$ is a real number with $p \geq n$, $E$ is a Banach space and $\mu \in W(B_{E'},w^*)$ is a given measure. Consider the continuous $n$-homogeneous polynomial
$$j_{p/n}^n \colon
C(B_{E'})\longrightarrow L_{p/n}(\mu)~,~j_{p/n}^n(f)=j_{p/n}(f^n).$$
  The restriction
of $j_{p/n}^n$ to $e(E)$ will be denoted $(j_{p/n}^e)^n$, so
$(j_{p/n}^e)^n \colon e(E) \longrightarrow E^{p/n}$ where
$E^{p/n}:={\rm SPAN}((j_{p/n}^e)^n\circ e(E)) \subseteq L_{p/n}(\mu)$, where ${\rm SPAN}$ denotes the linear hull. Define $G_{E\!,p}^\mu:=\overline{j_p^e\circ e (E)}\subset L_p(\mu)$.

 Let $\otimes^{n,s}_{\pi_s}E$  denote the $n$-fold symmetric tensor product of $E$ endowed with the
projective $s$-tensor norm $\pi_s$, and let $\widehat \otimes^{n,s}_{\pi_s}E$ denote its completion (see \cite{klaus} for
definitions and main properties). For simplicity, an elementary symmetric tensor $x \otimes \cdots \otimes x \in \otimes^{n,s}_{\pi_s}E$ shall be denoted by $\otimes^n x$. Given $P \in {\mathcal P}(^nE;F)$ let
$P_{L,s}$ be the linearization of $P$ , that is, $P_{L,s}$ is a linear
operator from $\widehat \otimes^{n,s}_{\pi_s}E$ into $F$ such that
$$P(x) = P_{L,s}(\otimes^n x) {\rm~for~every~} x \in E.$$

According to \cite{BoPeRu}, there is an injective linear operator
$\delta \colon \otimes^{n,s}_{\pi_s}E \longrightarrow C(B_{E'})$
such that $\delta (\otimes^n x)(x^*)= \langle x^*,x
\rangle^n$ for any $x\in E$ and $x^*\in B_{E'}$, and
$$j_{p/n}\circ\delta(\otimes^{n,s}_{\pi_s}E)={\rm SPAN}((j_{p/n}^e)^n\circ
e(E))= E^{p/n}
.$$ In order to introduce a convenient renorming on $E^{p/n}$, which is slightly different from the renorming defined in \cite{BoPeRu}, we introduce two auxiliary maps:\\
$\bullet$ By $T$ we denote the symmetric $n$-fold tensor product of
the linear operator $j_p\circ e$. So, $T \colon \otimes^{n,s}E \longrightarrow \otimes^{n,s} j_p \circ e(E)$ is a linear operator
satisfying
$$T(\otimes^n x ) = \otimes^n j_p(e(x)) {\rm~for~every~} x \in E.$$
$\bullet$ Consider the $n$-homogeneous polynomial
$$Q \colon  j_p(e(E))\subseteq L_p(\mu)\longrightarrow L_{p/n}(\mu)~,~ Q(g)=g^n,$$
and let $Q_{L,s}\colon \otimes_{n,s}^{\pi_s} j_p(e(E))\longrightarrow  L_{p/n}(\mu)$ denote the linearization of $Q$.

A norm $\pi_{p/n}$ is
defined on $E^{p/n}$ by:
$$
\pi_{p/n}((j_{p/n} \circ \delta)(\theta)) := \inf \left \{\sum_{i=1}^m |\lambda_i|\cdot\|(j_{p/n} \circ
\delta)(\otimes^n x_i)\|_{L_{p/n}}\right\},
$$
where the infimum is taken over all representations of $T(\theta) \in \otimes^{n,s}_{\pi_s}j_p\circ e(E)$ of the
form $T(\theta) = \sum_{i=1}^m \lambda_i \otimes^n j_p\circ e(x_i)$ with $m \in \mathbb{N}$,
$\lambda_i \in \mathbb{K}$ and $x_i \in E$. Our first task is to prove that this map is well-defined and is indeed a norm on $E^{p/n}$. The following lemma shall be useful in the next two results.

\begin{lemma}\label{lemanovo} $Q_{L,s} \circ T = j_{p/n}\circ \delta$.
\end{lemma}

\begin{proof} Given $k \in \mathbb{N}$, $x_1, \ldots, x_k \in E$, $\lambda_1,\ldots,\lambda_k$ scalars and $\varphi \in B_{E'}$,
\begin{align*}Q_{L,s} \circ T\left(\sum_{j=1}^k \lambda_j \otimes^n x_j  \right)(\varphi) &= Q_{L,s}\left(\sum_{j=1}^k \lambda_j \otimes^n j_p(e(x_j)) \right)(\varphi)\\&= \sum_{j=1}^k \lambda_jQ\left( j_p(e(x_j)) \right)(\varphi)=\sum_{j=1}^k\lambda_j \left[\left( j_p(e(x_j)) \right)\right]^n(\varphi)\\&=\sum_{j=1}^k \lambda_j\left[\left( j_p(e(x_j))\right)(\varphi)\right]^n=\sum_{j=1}^k\lambda_j \varphi(x_j)^n \\&=\sum_{j=1}^k \lambda_j \delta(\otimes^n x_j)(\varphi)= j_{\frac{p}{n}} \circ \delta\left(\sum_{j=1}^k \lambda_j \otimes^n x_j  \right)(\varphi).
\end{align*}
\end{proof}

\begin{proposition}{\rm \cite[Proposition 4.2]{BoPeRu}} $\pi_{\frac{p}{n}}$ is norm on $E^{\frac{p}{n}}$.
\end{proposition}

\begin{proof}
 Note first that  $\pi_{p/n}$ is well-defined. Indeed, if $T(\theta_1)=T(\theta_2)$, by Lemma \ref{lemanovo} we have
$$
j_{p/n}\circ \delta (\theta_1)=Q_{L,s}\circ T(\theta_1)=Q_{L,s}\circ T(\theta_2)=j_{p/n}\circ \delta (\theta_2).
$$
Let us see that $\pi_{p/n}(j_{p/n}\circ \delta(\theta))=0$ implies $j_{p/n}\circ \delta(\theta)=0$. Let $\epsilon>0$. If $\pi_{p/n}(j_{p/n}\circ \delta(\theta))=0$, then there is a representation $T(\theta) = \sum_{i=1}^m \lambda_i \otimes^n j_p\circ e(x_i)$ such that $$\sum_{i=1}^m|\lambda_i| \cdot \|j_{p/n}\circ \delta(\otimes^n x_i)\|_{L_{p/n}(\mu)}\leq \epsilon.$$ Then, again by Lemma \ref{lemanovo},
$$
j_{p/n}\circ \delta(\theta)=Q_{L,s}\circ T(\theta) = \sum_{i=1}^m \lambda_i Q_{L,s}(\otimes^n j_p\circ e(x_i) )= \sum_{i=1}^m \lambda_i j_{p/n}\circ \delta\circ \delta_m(x_i).
$$
Hence, for all $\phi\in B_{L_{p/n}(\mu)'}$,
$$
\left|\phi\circ j_{p/n}\circ \delta(\theta)\right|= \left|\sum_{i=1}^m\lambda_i \phi(j_{p/n}\circ \delta(\otimes^n x_i)\right| \leq \epsilon \|\phi\|.
$$
As $\epsilon>0$ is arbitrary, we conclude that $\phi\circ j_{p/n}\circ \delta(\theta) =0$ for all $\phi \in B_{L_{p/n}(\mu)'}$. As $p\geq n$ we can conclude that  $ j_{p/n}\circ \delta(\theta) =0$.

For every $x \in E$,
\begin{equation*}
\|j_p \circ e(x)\|_{L_p}^n = \left(\int_{B_{E'}}|\varphi(x)|^p d\mu(\varphi)\right)^{\frac{n}{p}} = \|j_{\frac{p}{n}} \circ \delta(\otimes^n x )\|_{L_{\frac{p}{n}}}.
\end{equation*}
So,  for $\theta \in \otimes^{n,s}E$ we have
\begin{align*}
\pi_s(T(\theta))&=\inf \left\{ \sum_{i = 1}^m |\lambda_i|\cdot \|j_p\circ e(x_i)\|_{L_p}^n :T(\theta) = \sum_{i=1}^m \lambda_i \otimes^n j_p\circ e(x_i) \right\}\\
 &= \inf \left\{ \sum_{i = 1}^m |\lambda_i|\cdot \|j_{\frac{p}{n}} \circ \delta(\otimes^n x_i )\|_{L_{\frac{p}{n}}} : T(\theta) = \sum_{i=1}^m \lambda_i \otimes^n j_p\circ e(x_i)
\right\}\\
 &= \pi_{\frac{p}{n}}(j_{\frac{p}{n}} \circ \delta(\theta)).
\end{align*}
It follows from the equality
\begin{equation}\label{iso} \pi_s(T(\theta)) = \pi_{\frac{p}{n}}(j_{\frac{p}{n}} \circ \delta(\theta)),
\end{equation}
from the linearity of the operators $j_{\frac{p}{n}}, \delta, T$ and from the fact that $\pi_s$ is a norm that $\pi_{\frac{p}{n}}$ fulfills the remaining norm axioms.
\end{proof}


\begin{theorem}{\rm \cite[Theorem 2.1]{BoPeRuLAMA}}\label{th} Let $E$ be a Banach space and $\mu$ be a regular Borel probability measure on $B_{E'}$. Then $\hat\otimes^{n,s}_{\pi_s} G_{E\!,p}^\mu$ is isometrically isomorphic to the completion $\hat E^{p/n}$ of the subspace $E^{p/n}$ of $L_{p/n}(\mu)$ with respect to the norm $\pi_{p/n}$.
\end{theorem}

\begin{proof} 
Calling on Lemma \ref{lemanovo} we have
$$Q_{L,s}( \otimes_{n,s}^{\pi_s} j_p(e(E)))= j_{p/n}\circ \delta\left( \otimes_{n,s}^{\pi_s}E\right)={\sc \rm SPAN}( j_{p/n}\circ \delta \circ \delta_n(E)) = E^{p/n},$$where {\sc \rm SPAN} denotes the linear hull.   Endowing $E^{p/n}$ with the norm $\pi_{\frac{p}{n}}$, by (\ref{iso}) it follows that $Q_{L,s}$ is a linear isometry from $\otimes_{n,s}^{\pi_s} j_p(e(E))$ onto $E^{p/n}$, which obviously induces an isometric isomorphism between their completions.
\end{proof}

\begin{proposition} {\rm \cite[Proposition 3.3]{BoPeRuLAMA}}\label{factorization2} A polynomial $P \in {\mathcal P}(^nE;F)$ is $p$-dominated if and only if there is a regular Borel probability measure $\mu$ on
$B_{E'}$ 
 and a continuous linear operator $v \colon \hat\otimes^{n,s}_{\pi_s}G_{E\!,p}^\mu
\longrightarrow F$ such that the following diagram commutes
\begin{center}
$\begin{CD}
E                      @ > P >>            F\\
@V{e}VV                  @ AAvA\\
e(E) @>{\delta^{G_{E\!,p}^\mu}_n\circ j^e_{p}}>>  \hat\otimes^{n,s}_{\pi_s}G_{E\!,p}^\mu\\
\end{CD}$
\end{center}
Moreover, $\|v\|=\|P\|_{d,p}.$
\end{proposition}

\begin{proof}
Assume that $P$ is $p$-dominated. By Proposition \ref{prop1} there is a differential Pietsch measure $\mu$ for $P$. Given $x\in E$, define $R(j_p^e\circ e(x)):=P(x)$. 
 By Proposition \ref{prop2},  $R$ is a well defined $n$-homogeneous polynomial from $j_p^e\circ e (E)\subset L_p(\mu)$ to $F$. The continuity of $R$ follows from (\ref{dominationpolynomial}). Then, the following
diagram commutes
\begin{center}
$\begin{CD}
E                      @ > P >>            F\\
@V{e}VV                  @ AARA\\
e(E) @>{j^e_{p}}>>  j_p^e\circ e (E)\\
@VVV @VVV\\
C(B_{E^*}) @>{j_{p}} >> L_p(\mu)
\end{CD}$
\end{center}
that is, $R\circ j_p^e\circ e=P$. Moreover, $\|R\|\leq \|P\|_{p,d}$. Defining $v:=R^L$ as the
linearization of $R$ on $\hat \otimes_{\pi_s}^{n,s}G_{E\!,p}^\mu$ we obtain the desired commutative diagram.

As to the converse, by \cite[Proposition 3.4]{BoPeRu} we know that $\delta_n^{G_{E\!,p}^\mu}\circ
j_p^e$ is $p$-dominated, so $P=v\circ(\delta_n^{G_{E\!,p}^\mu}\circ j_p^e)\circ e$ is $p$-dominated as
well. Denoting by $\pi_p(j_p^e)$ the $p$-summing norm of $j_p^e$, the following computation completes the
proof:
\begin{align*}
\|P\|_{d,p} & = \|R^L \circ \delta^{G_{E\!,p}^\mu}_n \circ j_p^e \circ e\|_{d,p}
\leq \|R^L\| \cdot \|  \delta^{G_{E\!,p}^\mu}_n \circ j_p^e
\|_{d,p} \cdot \|e\|^n\\
 & \leq  \|R^L\| \cdot \|\delta^{G_{E\!,p}^\mu}_n\| \cdot(\pi_p(j_p^e))^n \leq \|R^L\| = \|v\| =
\|R\| \leq \|P\|_{d,p}.
\end{align*}
\end{proof}

We now have all the ingredients to complete the proof of the factorization theorem for dominated polynomials \cite[Theorem 4.4]{BoPeRu}.

\begin{theorem}{\rm \cite[Theorem 4.4]{BoPeRu}}\label{Theorem 0.3} A polynomial $P \in {\mathcal P}(^nE;F)$ is $p$-dominated if and only if
there is a regular Borel probability measure $\mu$ on $B_{E'}$ with the weak* topology and a continuous
linear operator $u \colon \left(E^{\frac{p}{n}}, \pi_{\frac{p}{n}}\right) \longrightarrow F$ such that the
following diagram commutes
\begin{center}
$\begin{CD}
E                      @ > P >>            F\\
@V{\delta \,\, \circ \,\,\delta_n}VV                  @ AAuA\\
E_n @>{j^{E_n}_{\frac{p}{n}}}>> E^{\frac{p}{n}}\\
@VVV @VVV\\
C(B_{E'}) @>{j_{\frac{p}{n}}} >> L_{\frac{p}{n}(\mu)}
\end{CD}$
\end{center}
\end{theorem}

\begin{proof} Assume that $P$ is $p$-dominated. Let $\mu$ be the Borel probability and $v$ be the linear operator provided by Proposition \ref{factorization2}. In the proof of Theorem \ref{th} we saw that the operator
$$ Q_{L,s}\colon \otimes_{n,s}^{\pi_s} j_p(e(E))\longrightarrow  E^{\frac{p}{n}}$$
is an isometric isomorphism. Therefore the operator
$$u \colon  E^{\frac{p}{n}} \longrightarrow F~,~u:= v \circ Q_{L,s}^{-1},$$
is well-defined, linear and continuous. We just have to prove that the diagram commutes. Given $x \in E$,
\begin{align*} Q_{L,s}\left(\delta_n^{G^\mu_{E,p}}(j_p^e(e(x)))\right)(\varphi)&=
Q_{L,s}\left(j_p^e(e(x)) \otimes \cdots \otimes j_p^e(e(x) )\right)(\varphi)\\
&= Q(j_p^e(e(x)))(\varphi) = \left[j_p^e(e(x))(\varphi)\right]^n = \varphi(x)^n\\
&=\delta(x \otimes \cdots \otimes x)(\varphi) = \delta\left(\delta_n^E(x) \right)(\varphi)\\& = j_{\frac{p}{n}}\left(\delta\left(\delta_n^E(x) \right) \right)(\varphi)
\end{align*}
for every $\varphi \in B_{E'}$, proving that
$$Q_{L,s}\left(\delta_n^{G^\mu_{E,p}}(j_p^e(e(x)))\right) =  j_{\frac{p}{n}}\left(\delta\left(\delta_n^E(x) \right) \right), $$
hence
$$Q_{L,s}^{-1}\left( j_{\frac{p}{n}}\left(\delta\left(\delta_n^E(x) \right) \right)\right)=  \delta_n^{G^\mu_{E,p}}(j_p^e(e(x))). $$
Since the diagram of Proposition \ref{factorization2} commutes, we have
\begin{align*} P(x) & = v \left(\delta_n^{G^\mu_{E,p}}(j_p^e(e(x))) \right)= v \left(Q_{L,s}^{-1}\left( j_{\frac{p}{n}}\left(\delta\left(\delta_n^E(x) \right) \right)\right) \right) = u\left(j_{\frac{p}{n}}\left(\delta\left(\delta_n^E(x) \right) \right) \right).
\end{align*}
The proof of the converse is correct in \cite[Theorem 4.4]{BoPeRu}.
\end{proof}

As in \cite{BoPeRu} we can conclude then with the desired factorization theorem for dominated polynomials through the  dominated polynomial $\left(j^e_{\frac{p}{n}}\right)^n$.

\begin{theorem}{\rm \cite[Theorem 4.6]{BoPeRu}} \label{Theorem 0.6} A polynomial $P \in {\mathcal P}(^nE;F)$ is $p$-dominated if and only if
there is a regular Borel probability measure $\mu$ on $B_{E'}$ with the weak* topology and a continuous
linear operator $u \colon \left(E^{\frac{p}{n}}, \pi_{\frac{p}{n}}\right) \longrightarrow F$ such that the
following diagram commutes
\begin{center}
$\begin{CD}
E                      @ > P >>            F\\
@V{e}VV                  @ AAuA\\
e(E) @>{\left(j^e_{\frac{p}{n}}\right)^n}>> E^{\frac{p}{n}}\\
@VVV @VVV\\
C(B_{E'}) @>{\left(j_{\frac{p}{n}}\right)^n} >> L_{\frac{p}{n}(\mu)}
\end{CD}$
\end{center}
\end{theorem}

\bigskip

\noindent{\bf Acknowledgement.} The authors thank E. A. S\'anchez-P\'erez for his helpful suggestions.

\vspace{2mm}

%

\noindent[Geraldo Botelho] Faculdade de Matem\'atica, Universidade Federal de
Uberl\^andia, 38.400-902 -- Uberl\^andia, Brazil, e-mail: botelho@ufu.br.

\medskip

\noindent[Daniel Pellegrino] Departamento de Matem\'atica, Universidade Federal da Para\'iba, 58.051-900 -- Jo\~ao Pessoa, Brazil, e-mail: dmpellegrino@gmail.com.

\medskip

\noindent [Pilar Rueda] Departamento de An\'alisis Matem\'atico, Universidad de Valencia, 46.100 Burjasot -- Valencia, Spain, e-mail: pilar.rueda@uv.es.

\end{document}